\documentclass[12pt,leqno]{article}
\usepackage{amssymb}
\usepackage[mathscr]{eucal}
\usepackage{amsmath,amssymb,latexsym,theorem,enumerate,ifthen}
\usepackage[english]{babel}
\usepackage{color,url}
\usepackage{appendix}
\usepackage{graphicx}
\usepackage{subcaption}

\newcommand{\NN}{\mathbb{N}}

\newcommand{\RR}{\mathbb{R}}

\newcommand{\ZZ}{\mathbb{Z}}

\newcommand{\bx}{{\boldsymbol{x}}}

\newcommand{\by}{{\boldsymbol{y}}}

\newcommand{\cX}{{\mathcal X}}

\newcommand{\core}{\operatorname{cor}}

\newcommand{\comment}[1]{}

\renewcommand{\leq}{\leqslant}
\renewcommand{\geq}{\geqslant}

\newcommand{\proofend}{\hfill\mbox{$\Box$}}

\numberwithin{equation}{section}

\theoremstyle{change} \theorembodyfont{\em}
\newtheorem{Lem}{Lemma.}[section]
\newtheorem{Thm}[Lem]{Theorem.}

\theorembodyfont{\rm}
\newtheorem{Def}[Lem]{Definition.}
\newtheorem{Rem}[Lem]{Remark.}

\newtheorem{Ex}[Lem]{Example.}

\def\OnlyOnArXiv#1#2{\ifthenelse{\equal{#1}{Y}}{#2}{}}

\long\def\Eq#1#2{\ifthenelse{\equal{#1}{*}}
  {\begin{equation*}\begin{aligned}#2\end{aligned}\end{equation*}}
  {\begin{equation}\begin{aligned}\label{#1}#2\end{aligned}\end{equation}}}

\newenvironment{proof}{\noindent{\bf Proof.}}{\proofend}

\begin{document}

\begin{center}
 {\bfseries\Large Axiomatic characterisation of generalized $\psi$-estimators}

\vspace*{3mm}

{\sc\large
  M\'aty\'as $\text{Barczy}^{*,\diamond}$,
  Zsolt $\text{P\'ales}^{**}$ }

\end{center}

\vskip0.2cm

\noindent
 * HUN-REN–SZTE Analysis and Applications Research Group,
   Bolyai Institute, University of Szeged,
   Aradi v\'ertan\'uk tere 1, H--6720 Szeged, Hungary.

\noindent
 ** Institute of Mathematics, University of Debrecen,
    Pf.~400, H--4002 Debrecen, Hungary.

\noindent E-mails: barczy@math.u-szeged.hu (M. Barczy),
                  pales@science.unideb.hu  (Zs. P\'ales).

\noindent $\diamond$ Corresponding author.

\vskip0.2cm


{\renewcommand{\thefootnote}{}
\footnote{\textit{2020 Mathematics Subject Classifications\/}:
 {62A01, 62F10, 26E60} }
\footnote{\textit{Key words and phrases\/}:
generalized $\psi$-estimator, $Z$-estimator, characterisation, symmetry, internality, asymptotic idempotency.}
\vspace*{0.2cm}
}

\vspace*{-10mm}

\begin{abstract}
We give axiomatic characterisations of generalized $\psi$-estimators
 and (usual) $\psi$-estimators (also called $Z$-estimators), respectively.
The key properties of estimators that come into play in the characterisation theorems are the symmetry, the (strong) internality and the asymptotic idempotency.
In the proofs, a separation theorem for Abelian subsemigroups plays a crucial role.
\end{abstract}


\section{Introduction}
\label{section_intro}

In statistics, $M$-estimators play a fundamental role, and a special subclass, the class of $\psi$-estimators (also called $Z$-estimators) is in the heart of investigations as well.
In this paper, we address a foundational topic for a generalized version of $\psi$-estimators (see Definition \ref{Def_Tn}), namely, their axiomatic characterisation.
Results of this type are important in all branches of mathematics.
It will turn out that the key properties of estimators that come into play in the characterisation theorems are the symmetry, (strong) internality and asymptotic idempotency.
In the proofs, surprisingly, a separation theorem for Abelian subsemigroups (due to P\'ales {\cite[Theorem 1]{Pal89a}},
see also Theorem \ref{Thm_szeparacios}) plays a crucial role.

Let $(X,\cX)$ be a measurable space, $\Theta$ be a Borel subset of $\RR$, and $\psi:X\times\Theta\to\RR$ be a function such that for all $t\in\Theta$, the function $X\ni x\mapsto \psi(x,t)$ is measurable with respect to the sigma-algebra $\cX$.
Let $(\xi_n)_{n\geq 1}$ be a sequence of independent and identically distributed random variables with values in $X$ such that the distribution of $\xi_1$ depends on an unknown parameter $\vartheta \in\Theta$.
For each $n\geq 1$, Huber \cite{Hub64, Hub67} introduced, among others, an important estimator of $\vartheta$ based on the observations
 $\xi_1,\ldots,\xi_n$ as a solution $ \widehat\vartheta_{n,\psi}(\xi_1,\ldots,\xi_n)$ of the equation:
 \[
    \sum_{i=1}^n \psi(\xi_i,t)=0, \qquad t\in\Theta.
 \]
In the statistical literature, one calls $\widehat\vartheta_{n,\psi}(\xi_1,\ldots,\xi_n)$ a $\psi$-estimator of the unknown
 parameter $\vartheta\in\Theta$ based on the i.i.d.\ observations $\xi_1,\ldots,\xi_n$,
 while other authors call it a $Z$-estimator (the letter Z refers to ''zero'').
In fact, $\psi$-estimators are special $M$-estimators (where the letter $M$ refers to ''maximum likelihood-type'') that were also introduced by Huber \cite{Hub64, Hub67}.
For a detailed exposition of M-estimators and $\psi$-estimators, see, e.g., Kosorok \cite[Sections 2.2.5 and 13]{Kos} or van der Vaart \cite[Section 5]{Vaa}.

In our recent paper Barczy and P\'ales \cite{BarPal2}, we introduced the notion of weighted generalized $\psi$-estimators
(recalled below in Definition \ref{Def_Tn}), and we studied their existence and uniqueness.
Among others, given a function $\psi:X\times \Theta \to\RR$, we derived necessary as well as sufficient conditions under which there exists a unique generalized $\psi$-estimator based on any possible realization $(x_1,\ldots,x_n)\in X^n$, $n\geq 1$.
However, the following question remained open: given an arbitrary estimator for the unknown parameter $\vartheta\in\Theta$, can one find a function $\psi:X\times \Theta\to\RR$ such that the given estimator coincides with a generalized $\psi$-estimator based on any possible realization $(x_1,\ldots,x_n)\in X^n$, $n\geq 1$?
A similar question can be formulated for (usual) $\psi$-estimators ($Z$-estimators) as well.
This paper is devoted to answer these two questions, namely, to derive axiomatic characterisations of generalized $\psi$-estimators and (usual) $Z$-estimators, respectively.
We call the attention to the fact that, in our setup, it does not matter whether the realization $(x_1,\ldots,x_n)$ comes from i.i.d.\ random variables $\xi_1,\ldots,\xi_n$ or not.
In the investigation of the asymptotic properties of the generalized $\psi$ estimators based on $(\xi_1,\ldots,\xi_n)$ as $n\to\infty$, the property i.i.d.\ for the sequence $(\xi_k)_{k\geq 1}$ could play a certain role.

It will turn out that our results are somewhat similar to the well-known characterisation theorem of quasi-arithmetic means (that are generalizations of the sample mean), proved independently of each other by Kolmogorov \cite{Kol30}, Nagumo \cite{Nag30, Nag31}, and de Finetti \cite{Def31} (see also Tikhomirov \cite[page 144]{Tik}).
For completeness, we recall this result (in the spirit of Kolmogorov \cite{Kol30}) together with the notion of quasi-arithmetic means.
Our axiomatic characterisation of generalized $\psi$-estimators and (usual) $\psi$-estimators in Theorems \ref{Thm_charac_psi_est} and \ref{Thm_charac_cont_psi_est}
 are in fact natural counterparts of the characterisation of (strongly) internal means
 due to P\'ales \cite[Theorem 9]{Pal89a}, which served us as a motivation.

Throughout this paper, let $\NN$, $\ZZ_+$, $\RR$ and $\RR_+$ denote the sets of positive integers, non-negative integers, real numbers and  non-negative real numbers, respectively.
An interval $\Theta\subseteq\RR$ will be called nondegenerate if it contains at least two distinct points.
Given a nonempty set $S$ and a function $f:S\to\RR$, let $\inf f(S):=\inf\{f(s) : s\in S\}$ and $\sup f(S):=\sup\{f(s) : s\in S\}$.

A classical and well-studied class of means is the class of quasi-arithmetic means (see, e.g., the monograph of Hardy et al.\ \cite{HarLitPol34}).

\begin{Def}[Quasi-arithmetic mean]\label{Def_quasi_arithmetic}
Let $n\in\NN$, let $I$ be a nondegenerate interval  of $\RR$, and let $f:I\to\RR$ be a continuous and strictly increasing function.
The $n$-variable quasi-arithmetic mean $\mathscr{A}^f_n:I^n\to I$ is defined by
 \[
  \mathscr{A}^f_n(x_1,\ldots,x_n):= f^{-1} \bigg( \frac{1}{n} \sum_{i=1}^n f(x_i)\bigg), \qquad x_1,\ldots, x_n\in I,
 \]
where $f^{-1}$ denotes the inverse of $f$. The function $f$ is called the generator of $\mathscr{A}^f_n$.
\end{Def}

\begin{Thm}[Kolmogorov (1930), Nagumo (1930) and de Finetti (1931)]\label{Thm_Kolmogorov}
Let $I$ be a compact nondegenerate interval of $\RR$ and let $M_n:I^n\to\RR$, $n\in\NN$, be a sequence of functions. Then the following two statements are equivalent:
\vspace{-8pt}
\begin{enumerate}[(i)]\itemsep=-4pt
\item The sequence $(M_n)_{n\in\NN}$ is quasi-arithmetic, that is, there exists a continuous and strictly increasing function $f:I\to\RR$ such that
 \[
   M_n(x_1,\ldots,x_n) = \mathscr{A}^f_n(x_1,\ldots,x_n),
    \qquad x_1,\ldots,x_n\in I,\;\; n\in\NN.
 \]
\item The sequence $(M_n)_{n\in\NN}$ possesses the following properties:
\begin{itemize}
  \item  $M_n$ is continuous and strictly increasing in each variable for each $n\in\NN$,
  \item $M_n$ is symmetric for each $n\in\NN$  (i.e., $M_n(x_1,\ldots,x_n) = M_n(x_{\pi(1)},\ldots,x_{\pi(n)})$ for each
        $x_1,\ldots,x_n \in I$ and each permutation $(\pi(1),\ldots,\pi(n))$ of $(1,\ldots,n)$),
  \item $M_n(x_1,\ldots,x_n) = x$ whenever $x_1=\cdots =x_n = x\in I$, $n\in\NN$,
  \item $M_{n+m}(x_1,\ldots,x_n, y_1,\ldots,y_m) = M_{n+m}(\overline x_n, \ldots, \overline x_n, y_1,\ldots, y_m)$
        for each $n,m\in\NN$, $x_1,\ldots,x_n,y_1,\ldots,y_m\in I$, where $\overline x_n:=M_n(x_1,\ldots,x_n)$.
 \end{itemize}
\end{enumerate}
\end{Thm}

\begin{Rem}\label{Rem1}
(1).
The arithmetic, geometric and harmonic mean are quasi-arithmetic means corresponding to the generators $f:\RR\to \RR$, $f(x):=x$, $x\in\RR$;
 $f:(0,\infty)\to \RR$, $f(x):=\ln(x)$, $x>0$; and $f:(0,\infty)\to \RR$, $f(x) =-x^{-1}$, $x>0$, respectively.

(2).
The conditions in part (ii) of Theorem \ref{Thm_Kolmogorov} can be divided into three groups.
Namely, the continuity is a regularity-type condition, the strict increasingness can be considered as a condition of functional inequality-type, while the other three conditions are of functional equation-type.
\proofend
\end{Rem}

In what follows, we recall the basic concepts related to generalized $\psi$-estimators introduced in Barczy and P\'ales \cite{BarPal2}.

\begin{Def}\label{Def_sign_change}
Let $\Theta$ be a nondegenerate open interval of $\RR$. For a function $f:\Theta\to\RR$, consider the following three level sets
\[
  \Theta_{f>0}:=\{t\in \Theta: f(t)>0\},\qquad
  \Theta_{f=0}:=\{t\in \Theta: f(t)=0\},\qquad
  \Theta_{f<0}:=\{t\in \Theta: f(t)<0\}.
\]
We say that $\vartheta\in\Theta$ is a \emph{point of sign change (of decreasing type) for $f$} if
 \[
 f(t) > 0 \quad \text{for $t<\vartheta$,}
   \qquad \text{and} \qquad
    f(t)< 0 \quad  \text{for $t>\vartheta$.}
 \]
\end{Def}

Note that there can exist at most one element $\vartheta\in\Theta$ which is a point of sign change for $f$.
Further, if $f$ is continuous at a point $\vartheta$ of sign change, then $\vartheta$ is the unique zero of $f$.

Let $X$ be a nonempty set, $\Theta$ be a nondegenerate open interval of $\RR$.
Let $\Psi(X,\Theta)$ denote the class of real-valued functions $\psi:X\times\Theta\to\RR$ such that, for all $x\in X$, there exist $t_+,t_-\in\Theta$ such that $t_+<t_-$ and $\psi(x,t_+)>0>\psi(x,t_-)$.
Roughly speaking, a function $\psi\in\Psi(X,\Theta)$ satisfies the following property: for all $x\in X$,
 the function $\Theta \ni t\mapsto \psi(x,t)$ changes sign (from positive to negative) on the interval $\Theta$ at least once.

\begin{Def}\label{Def_Tn}
We say that a function $\psi\in\Psi(X,\Theta)$
  \begin{enumerate}[(i)]
    \item \emph{possesses the property $[C]$ (briefly, $\psi$ is a $C$-function)} if
           it is continuous in its second variable, i.e., if, for all $x\in X$,
           the mapping $\Theta\ni t\mapsto \psi(x,t)$ is continuous.
    \item \emph{possesses the property $[T_n]$ (briefly, $\psi$ is a $T_n$-function)
           for some $n\in\NN$} if there exists a mapping $\vartheta_{n,\psi}:X^n\to\Theta$ such that,
           for all $\pmb{x}=(x_1,\dots,x_n)\in X^n$ and $t\in\Theta$,
           \begin{align*}
             \psi_{\pmb{x}}(t):=\sum_{i=1}^n \psi(x_i,t) \begin{cases}
                 > 0 & \text{if $t<\vartheta_{n,\psi}(\pmb{x})$,}\\
                 < 0 & \text{if $t>\vartheta_{n,\psi}(\pmb{x})$},
            \end{cases}
           \end{align*}
          that is, for all $\pmb{x}\in X^n$, the value $\vartheta_{n,\psi}(\pmb{x})$ is a point of sign change for the function $\psi_{\pmb{x}}$. If there is no confusion, instead of $\vartheta_{n,\psi}$ we simply write $\vartheta_n$.
          We may call $\vartheta_{n,\psi}(\pmb{x})$ as a generalized $\psi$-estimator for
         some unknown parameter in $\Theta$ based on the realization $\bx=(x_1,\ldots,x_n)\in X^n$. If, for each $n\in\NN$, $\psi$ is a $T_n$-function, then we say that \emph{$\psi$ possesses the property $[T]$ (briefly, $\psi$ is a $T$-function)}.
    \item \emph{possesses the property $[Z_n]$ (briefly, $\psi$ is a $Z_n$-function) for some $n\in\NN$} if it is a $T_n$-function and
    \[
   \psi_{\pmb{x}}(\vartheta_{n,\psi}(\pmb{x}))=\sum_{i=1}^n \psi(x_i,\vartheta_{n,\psi}(\pmb{x}))= 0
    \qquad \text{for all}\quad \pmb{x}=(x_1,\ldots,x_n)\in X^n.
    \]
    If, for each $n\in\NN$, $\psi$ is a $Z_n$-function, then we say that \emph{$\psi$ possesses the property $[Z]$ (briefly, $\psi$ is a $Z$-function)}.
   \end{enumerate}
\end{Def}

It can be seen that if $\psi$ is continuous in its second variable, and, for some $n\in\NN$, it is a $T_n$-function, then it also a $Z_n$-function.
Similarly, if $\psi$ possesses the properties $[C]$ and $[T]$, then it possesses the property $[Z]$ as well.
Furthermore, if $\psi\in\Psi(X,\Theta)$ is a $T_n$-function for some $n\in\NN$, then $\vartheta_{n,\psi}$ is symmetric in the sense that
$\vartheta_{n,\psi}(x_1,\ldots,x_n) = \vartheta_{n,\psi}(x_{\pi(1)},\ldots,x_{\pi(n)})$ holds for all $x_1,\ldots,x_n\in X$ and all permutations $(\pi(1),\ldots,\pi(n))$ of $(1,\ldots,n)$. (This follows from the fact that $\sum_{i=1}^n \psi(x_i,t) = \sum_{i=1}^n \psi(x_{\pi(i)},t)$, $t\in\Theta$.)

In Barczy and P\'ales \cite[Proposition 1]{BarPal2}, we proved that if $(X,\cX)$ is a measurable space,
 $n\in\NN$, $\psi\in\Psi[Z_n](X,\Theta)$, and $\psi$ is measurable in its first variable, then $\vartheta_{n,\psi}:X^n\to \Theta$ is measurable with respect to the sigma-algebras $\cX^n$.

Given $q\in\NN$ and properties $[P_1], \ldots, [P_q]$ belonging to the family of properties
\[
 \big\{[C],[T],[Z]\big\} \cup\big\{[T_n], [Z_n]\colon n\in\NN\big\},
\]
the subclass of $\Psi(X,\Theta)$ consisting of elements possessing the properties $[P_1],\ldots$, $[P_q]$ will be denoted by $\Psi[P_1,\ldots,P_q](X,\Theta)$, i.e.,
\[
  \Psi[P_1,\ldots,P_q](X,\Theta)
  :=\bigcap_{i=1}^q\Psi[P_i](X,\Theta).
\]

The paper is structured as follows.
In Section \ref{Sect1}, we provide a characterisation theorem for generalized $\psi$-estimators, see Theorem \ref{Thm_charac_psi_est}. The key properties of estimators that come into play in the characterisation theorems are the symmetry, the (strong) internality and the asymptotic idempotency, which are introduced in Section \ref{Sect1}.
For an interpretation of these properties from a statistical point of view, see Remark \ref{Rem_properties_interpret}.
In the proof of Theorem \ref{Thm_charac_psi_est}, a separation theorem for Abelian subsemigroups (due to P\'ales \cite[Theorem 1]{Pal89a}) plays a crucial role, of which the application is a novelty in statistics according to our knowledge, and for completeness, we also recall this result, see Theorem \ref{Thm_szeparacios}. We also use Theorem 2.1 and Corollary 3.2 in Barczy and P\'ales \cite{BarPal4} in the proof of Theorem \ref{Thm_charac_psi_est}, which state the internality and asymptotic idempotency of a generalized $\psi$-estimator, respectively.
In Example \ref{Exl_1}, we provide some applications of Theorem \ref{Thm_charac_psi_est}.
In Section \ref{Sect2}, we give a characterisation theorem for (usual) $\psi$-estimators ($Z$-estimators)
corresponding to a function $\psi\in\Psi(X,\Theta)$, which possesses the property $[C]$.

\section{Characterisation theorem for generalized $\psi$-estimators}
\label{Sect1}

In this section, we give an axiomatic characterisation of generalized $\psi$-estimators.
First, we recall a separation theorem for Abelian subsemigroups that plays a crucial role in the proof.

Let $(S,\oplus)$ be an Abelian semigroup, and $A\subseteq S$ be a subsemigroup.

\begin{Rem}
An Abelian semigroup $(S,\oplus)$ without a neutral element can be embedded into an Abelian semigroup with a neutral element.
Indeed, if $(S,\oplus)$ is an Abelian semigroup without a neutral element, then, for any symbol $\Delta\notin S$, one can extend the binary operation $\oplus$ on $S$ to a binary
 operation $\oplus_\Delta$ on $S\cup\{\Delta\}$ by letting $\Delta\oplus_\Delta s:=s\oplus_\Delta \Delta:=s$,
 $s\in S\cup\{\Delta\}$.
Then $(S\cup\{\Delta\},\oplus_\Delta)$ is an Abelian semigroup with a neutral element $\Delta$, and $(S,\oplus)$ is embedded into it.
As a consequence, in the forthcoming definitions and results,  if one prefers, then, without loss of generality, one can assume that an Abelian semigroup admits a neutral element.
\proofend
\end{Rem}

\begin{Def}
Let $(S,\oplus)$ be an Abelian semigroup, and $A\subseteq S$ be a subsemigroup.
The (algebraic) core of $A$ is defined as the subsemigroup
 \[
   \core(A):=\{a\in A : \text{$\forall\, s\in S$ \ $\exists\;n\in\NN$ \ such that $na\oplus s\in A$}\},
 \]
 where $na:=a\oplus\cdots\oplus a$ with $n$ terms on the right hand side.
\end{Def}

One can indeed check that, if $A$ is a subsemigroup of $S$, then $\core(A)$ is also a subsemigroup of $S$, morever, it also holds that $\core(A) \oplus A\subseteq \core(A)$.

Now, we recall a separation theorem for Abelian subsemigroups due to P\'ales \cite[Theorem 1]{Pal89a}.

\begin{Thm}[P\'ales {\cite[Theorem 1]{Pal89a}}]\label{Thm_szeparacios}
Let $(S,\oplus)$ be an Abelian semigroup, $A$ and $B$ are disjoint subsemigroups of $S$ such that $\core(A)\ne \emptyset$ and $\core(B)\ne \emptyset$.
Then there exists a homomorphism $F:S\to\RR$ (i.e., $F(s_1\oplus s_2)=F(s_1)+F(s_2)$, $s_1,s_2\in S$) such that
 \[
   F(a)\geq 0 \geq F(b),\qquad a\in A,\;\; b\in B,
 \]
 and
 \[
  F(a)> 0 > F(b),\qquad a\in \core(A),\;\; b\in \core(B).
 \]
\end{Thm}

Next, we recall two properties of a real-valued function defined on an Abelian semigroup that will have a key role in the forthcoming Theorem \ref{Thm_charac_psi_est}.

\begin{Def}
Let $(S,\oplus)$ be an Abelian semigroup.
We say that a function $f:S\to\RR$ is \emph{internal} if, for all $r,s\in S$,
\Eq{*}{
  \min(f(r),f(s))\leq f(r\oplus s)\leq \max(f(r),f(s))
}
holds.
If both inequalities are strict whenever $r,s\in S$ and $f(r)\neq f(s)$, then $f$ is called \emph{strictly internal}.
\end{Def}

If $f:S\to\RR$ is an internal function, then $f$ is also idempotent, that is, the equality $f(nr)=f(r)$ holds for all $n\in\NN$ and $r\in S$. (A simple proof can be obtained by induction.)
This motivates the following definition.

\begin{Def}\label{Def_asymp_idem}
Let $(S,\oplus)$ be an Abelian semigroup.
We say that a function $f:S\to\RR$ is  \emph{asymptotically idempotent} if
\Eq{*}{
  \lim_{n\to\infty} f(nr\oplus s)=f(r) \qquad
  \text{ for all $r,s\in S$.}
 }
\end{Def}

A notion that is similarly termed as the above defined property 'asymptotically idempotent'
appears in Ricci \cite[Definition 8]{Ric}, where a sequence $f_n:[a,b]^n\to\RR$, $n\in\NN$ (where $a<b$, $a,b\in\RR$) of aggregation functions (see Ricci \cite[Definition 1]{Ric}) is called asymptotically idempotent if
 \[
   \lim_{n\to\infty} f_n\big(\underbrace{x,\ldots,x}_{n}\big)=f(x) \qquad \text{for all $x\in[a,b]$.}
 \]

Given a function $M:\bigcup_{n=1}^\infty X^n\to\Theta$ and $m\in\NN$, we will denote by $M_m$ the restriction of $M$ onto $X^m$.
In the next theorem, we give an axiomatic characterisation of generalized $\psi$-estimators.

\begin{Thm}\label{Thm_charac_psi_est}
Let $X$ be a nonempty set, $\Theta$ be a nondegenerate open interval of $\RR$, and $M:\bigcup_{n=1}^\infty X^n\to\Theta$ be a function  such that $\inf M_1(X)=\inf \Theta$ and $\sup M_1(X)=\sup \Theta$.
Then the following two statements are equivalent:
 \vspace{-8pt}
 \begin{enumerate}[(i)]\itemsep=-4pt
  \item There exists a function $\psi\in\Psi[T](X,\Theta)$ such that, for all $n\in\NN$ and $x_1,\ldots,x_n\in X$, it holds that
  \Eq{*}{
  \vartheta_{n,\psi}(x_1,\ldots,x_n)=M_n(x_1,\ldots,x_n).
  }
 \item The function $M$ possesses the following properties:
       \begin{itemize}
         \item[(a)] Symmetry: $M_n$ is symmetric for each $n\in\NN$,
                i.e., 
                \[
                  M_n(x_1,\ldots,x_n) = M_n(x_{\pi(1)},\ldots,x_{\pi(n)})
                \]
                for all $x_1,\ldots,x_n \in X$ and each permutation $(\pi(1),\ldots,\pi(n))$ of $(1,\ldots,n)$,
         \item[(b)] Internality (mean-type property): for each $n,k\in\NN$ and $(x_1,\ldots,x_n)\in X^n$, $(y_1,\ldots,y_k)\in X^k$, we have
               \begin{align*}
                   \min(M_n(x_1,\ldots,x_n),M_k(y_1,\ldots,y_k))
                        &\leq M_{n+k}(x_1,\ldots,x_n,y_1,\ldots,y_k)\\
                        &\leq \max(M_n(x_1,\ldots,x_n),M_k(y_1,\ldots,y_k)),
               \end{align*}
       \item[(c)] Asymptotic idempotency:
          for all $k\in\NN$, and $x_1,\dots,x_k,y\in X$,
        \[
          \lim_{n\to\infty}M_{kn+1}(\underbrace{x_1,\ldots,x_1}_{n},\dots,\underbrace{x_k,\ldots,x_k}_{n},y) = M_k(x_1,\dots,x_k).
        \]
       \end{itemize}
 \end{enumerate}
\end{Thm}

\begin{Rem}\label{Rem_properties_interpret}
(1). The properties (a), (b) and (c) in part (ii) of Theorem \ref{Thm_charac_psi_est} may be interpreted from a statistical point of view as follows.

The symmetry property (a) means that a generalized $\psi$-estimator based on the realization $(x_1,\ldots,x_n)\in X^n$ does not depend on the order of the realization in question, only the values of $x_1,\ldots,x_n$ are important, but not their order in which we observe them.
 
The internality property (b) reflects the following somewhat natural requirement for an estimator.
Supposing that one estimates an unknown parameter based on the realizations $(x_1,\ldots,x_n)\in X^n$
and $(y_1,\ldots,y_k)\in X^k$, respectively, then it is natural to expect that if one estimates the unknown parameter based on the joint realization
$(x_1,\ldots,x_n,y_1,\ldots,y_k)\in X^{n+k}$, then this later estimate should be in the convex hull of the two original estimates.
In Example \ref{Exl_1}, we also highlight the role of 
the internality property (b) in Theorem \ref{Thm_charac_psi_est}. Roughly speaking, the estimate based on the joint realization cannot be worse than those based on the sub-realizations.
 
The asymptotic idempotency property (c) can be interpreted as follows.
Given $n,k\in\NN$, suppose that one observes $(x_1,\ldots,x_k)\in X^k$ repeatedly $n$-times and then $y\in X$.
The property (c) means that the generalized $\psi$-estimator does not essentially depend on $y\in X$ in the sense that its effect disappears as $n\to\infty$.
One could say that, in this scenario, the observation $y$ can be considered asymptotically as an outlier.

(2). A possible direction for future research is to explore the extension of our setup and Theorem \ref{Thm_charac_psi_est} from a one-dimensional parameter set $\Theta$ to a multidimensional one (note that, in our present setup, $\Theta$ is supposed to be a nondegenerate open interval of $\RR$).
\proofend      
\end{Rem}

\noindent{\bf Proof of Theorem \ref{Thm_charac_psi_est}.}
First, assume that assertion (i) holds.
Then $M$ is symmetric (see the paragraph after Definition \ref{Def_Tn}) and is internal (see Theorem 2.1 in Barczy and P\'ales \cite{BarPal4}).
The asymptotic idempotency of $M$ follows from Corollary 3.2 in Barczy and P\'ales \cite{BarPal4}.
Indeed, with the notation $\bx:=(x_1,\ldots,x_k)\in X^k$, Corollary 3.2 in Barczy and P\'ales \cite{BarPal4} implies that
 \[
   \lim_{n\to\infty} \vartheta_{kn+1,\psi}(\underbrace{\bx,\ldots,\bx}_{n},y) =  \vartheta_{k,\psi}(\bx), \qquad y\in X.
 \]
This, together with the symmetry of $\vartheta_{kn+1,\psi}$, yield the asymptotic idempotency of $M$.

Next, assume that assertion (ii) holds.
Let $(S(X),\oplus)$ be the free Abelian semigroup generated by the elements of $X$, i.e., $S(X)$ consists of all the finite (unordered) sequences (or unordered strings)
 of $X$ having positive lengths, furnished with the string concatenation $\oplus$ as an operation.
Note that, for all $n\in\NN$, $x_1,\ldots,x_n \in X$ and each permutation $(\pi(1),\ldots,\pi(n))$ of $(1,\ldots,n)$, the (unordered) sequences $x_1\oplus\cdots\oplus x_n$ and $x_{\pi(1)}\oplus\cdots\oplus x_{\pi(n)}$ coincide.

Let us introduce the mapping $\mu:S(X)\to \Theta$ by
 \[
   \mu(x_1\oplus\cdots\oplus x_n):=M_n(x_1,\ldots,x_n), \qquad n\in\NN, \;\; x_1,\ldots,x_n\in X.
 \]
Due to the symmetry of $M$, the mapping $\mu$ is well-defined, since for all $n\in\NN$, $x_1,\ldots,x_n \in X$ and permutation $(\pi(1),\ldots,\pi(n))$ of $(1,\ldots,n)$,
 we have
 \[
  \mu(x_{\pi(1)}\oplus\cdots\oplus x_{\pi(n)})
     = M(x_{\pi(1)},\ldots,x_{\pi(n)})
     = M(x_1,\ldots,x_n)
     = \mu(x_1\oplus\cdots\oplus x_n),
 \]
i.e., the mapping $\mu$ takes the same value on the (unordered) sequences $x_1\oplus\cdots\oplus x_n$ and $x_{\pi(1)}\oplus\cdots\oplus x_{\pi(n)}$  (that are the same elements of $S(X)$).

Let $t\in\Theta$ be fixed and define
 \[
 A_t:=\{ s\in S(X) : \mu(s) < t\}, \qquad \qquad B_t:=\{ s\in S(X) : \mu(s) > t\}.
 \]
We verify the conditions of Theorem \ref{Thm_szeparacios} with the choices $S:=S(X)$, $A:=A_t$ and $B:=B_t$.

Since $\Theta$ is open, $\inf \mu(X)=\inf M_1(X)=\inf\Theta$ and $\sup \mu(X)=\sup M_1(X)=\sup\Theta$, we can see that neither $A_t$ nor $B_t$ is empty.

Using that $M$ is internal, it follows that $\mu$ also possesses the following internality property:
\begin{align}\label{help_m_intern}
  \min(\mu(r),\mu(s))
    \leq \mu(r\oplus s)
    \leq  \max(\mu(r),\mu(s)), \qquad r,s\in S(X).
\end{align}
To prove the left hand side inequality, let $r,s\in S(X)$ be arbitrary. Then, there exist $n,k\in\NN$ and $x_1,\dots,x_n,y_1,\dots,y_k\in X$ such that $r=x_1\oplus\cdots\oplus x_n$ and $s=y_1\oplus\cdots\oplus y_k$. In view of the internality of $M$, we get
\Eq{*}{
 \min(\mu(r),\mu(s))
 &=\min(\mu(x_1\oplus\cdots\oplus x_n),\mu(y_1\oplus\cdots\oplus y_k))\\
 &=\min\big(M_n(x_1,\dots,x_n),M_k(y_1,\dots,y_k)\big)\\
 &\leq M_{n+k}(x_1,\dots,x_n,y_1,\dots,y_k)\\
 &=\mu(x_1\oplus\cdots\oplus x_n\oplus y_1\oplus\cdots\oplus y_k)\\&=\mu(r\oplus s).
}
The proof of the right hand side inequality in \eqref{help_m_intern} is completely analogous.

If $r,s\in A_t$, then $\mu(r)<t$ and $\mu(s)<t$, and \eqref{help_m_intern} yields that $\mu(r\oplus s)< t$, i.e., $r\oplus s\in A_t$.
Consequently, $A_t$ is a subsemigroup of $S(X)$.
A similar argument shows that $B_t$ is a subsemigroup of $S(X)$ as well.

The subsemigroups $A_t$ and $B_t$ are disjoint, since otherwise, there would be an element $s\in S(X)$ such that $\mu(s)<t < \mu(s)$, leading us to an obvious contradiction.

Finally, applying the asymptotic idempotency of $M$, we show that $\core(A_t) = A_t$ and $\core(B_t) = B_t$, yielding, in particular, that $\core(A_t)\ne \emptyset$
 and $\core(B_t)\ne\emptyset$. To prove that $A_t\subseteq \core(A_t)$, let $a\in A_t$ and $s\in S(X)$ be arbitrary.
Then there exist $k,\ell\in\NN$ and $x_1,\dots,x_k,y_1,\dots,y_\ell\in X$ such that $a=x_1\oplus\cdots\oplus x_k$ and $s=y_1\oplus\cdots\oplus y_\ell$.
Then $\mu(a)<t$, and, by the symmetry and the asymptotic idempotency of $M$, for all $i\in\{1,\dots,\ell\}$, we have
 \[
   \lim_{n\to\infty} \mu(na\oplus y_i)
       = \lim_{n\to\infty}M_{kn+1}(\underbrace{x_1,\ldots,x_1}_{n},\dots,\underbrace{x_k,\ldots,x_k}_{n},y_i) = M_k(x_1,\dots,x_k) = \mu(a)<t.
 \]
Therefore, for all $i\in\{1,\dots,\ell\}$, there exists $n_i\in\NN$ such that
\Eq{*}{
  \mu(n_ia\oplus y_i)<t.
}
Using that $(n_1+\dots+n_\ell)a\oplus s = (n_1a\oplus y_1)\oplus (n_2a\oplus y_2) \oplus \cdots \oplus (n_\ell a\oplus y_\ell)$
 and the internality property of $\mu$ established in \eqref{help_m_intern}, we can get that
\Eq{*}{
  \mu((n_1+\dots+n_\ell)a\oplus s)\leq
  \max_{1\leq i\leq\ell}\mu(n_ia\oplus y_i)<t.
}
This implies that $na\oplus s\in A_t$ with $n:=n_1+\dots+n_\ell$, therefore $a\in\core(A_t)$, which completes the proof of the inclusion $A_t\subseteq \core(A_t)$.
Using that $\core(A_t)\subseteq A_t$ also holds (due to the definition of $\core(A_t)$), we obtain the equality $\core(A_t) = A_t$, as desired.
A similar argument shows that $\core(B_t) = B_t$.

Applying Theorem \ref{Thm_szeparacios}, for all $t\in\Theta$, there exists a homomorphism $F_t:S(X)\to \RR$ such that
 \Eq{sep}{
   F_t(a) < 0 < F_t(b),\qquad a\in A_t,\;\; b\in B_t.
 }
Let us define $\psi:X\times \Theta \to\RR$ by the equality $\psi(x,t):= F_t(x)$, $x\in X$, $t\in\Theta$. Using that $F_t$ is a homomorphism, for all $t\in\Theta$, $n\in\NN$ and $x_1,\ldots,x_n\in X$, we have
 \[
  \sum_{i=1}^n \psi(x_i,t) = \sum_{i=1}^n F_t(x_i) = F_t(x_1\oplus \cdots \oplus x_n).
 \]
Hence, in view of \eqref{sep}, for all $t\in\Theta$, $n\in\NN$ and $x_1,\ldots,x_n\in X$, we get that
 \[
 \sum_{i=1}^n \psi(x_i,t) > 0 \qquad \text{if \ $\mu(x_1\oplus \cdots\oplus x_n)>t$,}
 \]
 and
 \[
  \sum_{i=1}^n \psi(x_i,t) < 0 \qquad \text{if \ $\mu(x_1\oplus \cdots\oplus x_n)<t$,}
 \]
 yielding that $\psi\in\Psi[T](X,\Theta)$ such that $\vartheta_{n,\psi}(x_1,\ldots,x_n) = \mu(x_1\oplus \cdots\oplus x_n) = M_n(x_1,\ldots,x_n)$,
 as desired.
\proofend

\begin{Rem}
(1). Note that for the implication $(i)\Longrightarrow(ii)$ in Theorem \ref{Thm_charac_psi_est},
 the assumptions $\inf M_1(X)=\inf \Theta$ and $\sup M_1(X)=\sup \Theta$ were not needed.
Further, observe that, in the implication $(ii)\Longrightarrow(i)$ in Theorem \ref{Thm_charac_psi_est},
 we only established the existence of an appropriate function $\psi$ corresponding to a given function $M$,   
 but we did not provide any explicit method how one can construct such a $\psi$.
The reason for it is that, in the corresponding proof, we use a separation theorem for Abelian subsemigroups (see Theorem \ref{Thm_szeparacios}), which, in general, does not provide any explicit construction for a homomorphism which separates the two Abelian subsemigroups.

(2).
Conditions (a), (b) and (c) in part (ii) of Theorem \ref{Thm_charac_psi_est} are of the type of functional equation, functional inequality and regularity, respectively.
These three types of conditions also appeared in the characterisation theorem of quasi-arithmetic means, see Theorem \ref{Thm_Kolmogorov}
 and part (2) of Remark \ref{Rem1}.
\proofend
\end{Rem}

In the next example, we provide some applications of Theorem \ref{Thm_charac_psi_est}.

\begin{Ex}\label{Exl_1}
(1). Let $\alpha>0$ and let $\xi$ be an absolutely continuous random variable with a density function
   \begin{align*}
        f_\xi(x):=\begin{cases}
                2\alpha x (1-x^2)^{\alpha-1}  & \text{if $x\in(0,1)$,}\\
                0 & \text{otherwise.}
   \end{cases}
  \end{align*}
Then, one can check that given an $n\in\NN$ and a realization $x_1,\ldots,x_n\in(0,1)$
 of a sample of size $n$ for $\xi$, there exists a unique MLE of $\alpha$ based on $x_1,\ldots,x_n$,
 and it takes the form
 \begin{align*}
  -\frac{n}{\sum_{i=1}^n \ln(1-x_i^2)}
  = -\frac{1}{\ln\left( \prod_{i=1}^n (1-x_i^2)^{\frac{1}{n}} \right)}. 
 \end{align*}
One can easily check that the above MLE of $\alpha$ is, in fact, a $\psi$-estimator corresponding to the function $\psi:(0,1)\times(0,\infty)\to\RR$,
 \[
     \psi(x,\alpha) = \frac{1}{\alpha} + \ln(1-x^2), \qquad x\in(0,1),\quad \alpha>0.
 \]
Since 
 \[
   \inf\left\{-\frac{1}{\ln(1-x^2)} : x\in(0,1) \right\}
      = 0 =\inf(\RR_{++})
 \]   
 and
 \[
   \sup\left\{-\frac{1}{\ln(1-x^2)} : x\in(0,1) \right\}
      = \infty = \sup(\RR_{++}),
 \]   
 the conditions of Theorem \ref{Thm_charac_psi_est} are satisfied for $X:=(0,1)$, $\Theta:=\RR_{++}$
 and for the estimator $M:\bigcup_{n=1}^\infty (0,1)^n \to \RR_{++}$,
 \[
   M(x_1,\ldots,x_n):= -\frac{1}{\ln\left( \prod_{i=1}^n (1-x_i^2)^{\frac{1}{n}} \right)},
      \qquad x_1,\ldots,x_n\in(0,1),\qquad n\in\NN.
 \]  
Therefore, by Theorem \ref{Thm_charac_psi_est}, the properties (a), (b) and (c) in part (ii)
 of this theorem hold for $M$.
In what follows, we directly (without using Theorem \ref{Thm_charac_psi_est}) verify that these properties indeed hold.
 
The symmetry property (a) trivially holds, since $\prod_{i=1}^n (1-x_i^2) = \prod_{i=1}^n (1-x_{\pi(i)}^2)$
 for each permutation $(\pi(1),\ldots,\pi(n))$ of $(1,\ldots,n)$.
   
To check the internality property (b), let $n,k\in\NN$, $(x_1,\ldots,x_n)\in (0,1)^n$ and $(y_1,\ldots,y_k)\in (0,1)^k$.
Supposing that $M_n(x_1,\ldots,x_n)\leq M_k(y_1,\ldots,y_k)$, we need to check that 
 \begin{align*}
   -\frac{1}{\ln\left( \prod_{i=1}^n (1-x_i^2)^{\frac{1}{n}} \right)}
     \leq -\frac{1}{\ln\left( \left(\prod_{i=1}^n (1-x_i^2) \prod_{j=1}^k (1-y_j^2)\right)^{\frac{1}{n+k}}  \right)}
      \leq  -\frac{1}{\ln\left( \prod_{j=1}^k (1-y_j^2)^{\frac{1}{k}} \right)},
 \end{align*}
 which is equivalent to
 \begin{align*} 
   \ln\left( \prod_{i=1}^n (1-x_i^2)^{\frac{1}{n}} \right)
     \leq \ln\left( \left(\prod_{i=1}^n (1-x_i^2) \prod_{j=1}^k (1-y_j^2)\right)^{\frac{1}{n+k}}  \right)
     \leq \ln\left( \prod_{j=1}^k (1-y_j^2)^{\frac{1}{k}} \right),
 \end{align*}
  i.e., by the strict increasingness of $\ln$,
 \begin{align*} 
    \prod_{i=1}^n (1-x_i^2)^{\frac{1}{n}} 
     \leq  \left(\prod_{i=1}^n (1-x_i^2) \prod_{j=1}^k (1-y_j^2)\right)^{\frac{1}{n+k}}
     \leq  \prod_{j=1}^k (1-y_j^2)^{\frac{1}{k}}.
 \end{align*}
This is equivalent that the following two inequalities hold
 \[
     \prod_{i=1}^n (1-x_i^2)^{\frac{k}{n(n+k)}}
        \leq  \prod_{j=1}^k (1-y_j^2)^{\frac{1}{n+k}}
        \qquad \text{and}\qquad
      \prod_{i=1}^n (1-x_i^2)^{\frac{1}{n+k}}
           \leq   \prod_{j=1}^k (1-y_j^2)^{\frac{n}{k(n+k)}},
 \]  
  which are both equivalent to the inequality 
  $M_n(x_1,\ldots,x_n)\leq M_k(y_1,\ldots,y_k)$ being satisfied.
  
To check the asymptotic idempotency property (c), let $k\in\NN$ and $x_1,\ldots,x_k,y\in(0,1)$.
Then
 \begin{align*}
  &\lim_{n\to\infty}M_{kn+1}(\underbrace{x_1,\ldots,x_1}_{n},\dots,\underbrace{x_k,\ldots,x_k}_{n},y)\\
  &\qquad= \lim_{n\to\infty} \frac{-1}{\ln\left(\left( \prod_{i=1}^k (1-x_i^2)^n (1-y^2) \right)^{\frac{1}{nk+1}}  \right) }\\
  &\qquad = \lim_{n\to\infty}   
               \frac{-1}{ \frac{n}{nk+1} \ln\left( \prod_{i=1}^k (1-x_i^2)\right) + \frac{1}{nk+1}\ln(1-y^2) }
\end{align*}
\begin{align*}               
 &=  \frac{-1}{\frac{1}{k}\ln\left( \prod_{i=1}^k (1-x_i^2) \right)}               
   = M_k(x_1,\dots,x_k),
 \end{align*}
 as desired.  
   
(2).
Let $X:=\Theta:=\RR_{++}$ and define an estimator $M:\bigcup_{n=1}^\infty\RR_{++}^n\to\RR_{++}$ by
\[
  M(x_1,\dots,x_n):=\frac1{2n}\Big(x_1+\dots+x_n+n\sqrt[n]{x_1\cdots x_n}\Big),
  \qquad n\in\NN,\,\,x_1,\dots,x_n\in\RR_{++}.
\]
Since $M_1(x) = x$, $x\in\RR_{++}$, we have that $M_1(\RR_{++}) = \RR_{++}$, yielding that 
 $\inf(M_1(\RR_{++})) = 0 =\inf(\RR_{++})$ and $\sup(M_1(\RR_{++})) = \infty = \sup(\RR_{++})$.
Further, the internality property (b) of part (ii) in Theorem \ref{Thm_charac_psi_est} does not hold for $M$,
 since for the realization $(x_1,x_2,x_3,x_4):=(1,81,25,25)$, it does not hold that
 \[
  \min(M_2(1,81),M_2(25,25))\leq M_4(1,81,25,25)\leq \max(M_2(1,81),M_2(25,25)),
\]
since we have that $M_2(1,81)=M_2(25,25)=25$ and $M_4(1,81,25,25)=24$.
Therefore, by Theorem \ref{Thm_charac_psi_est}, there does not exist a $\psi\in\Psi[T](\RR_{++},\RR_{++})$ such that
 \[
  M(x_1,\dots,x_n)=\vartheta_{n,\psi}(x_1,\dots,x_n),\qquad n\in\NN,\,\,x_1,\dots,x_n\in\RR_{++},
 \]
 i.e., $M$ cannot be a generalized $\psi$-estimator.
\proofend
\end{Ex}

\section{Characterisation theorem for (usual) $\psi$-estimators }
\label{Sect2}

In this section, we give an axiomatic characterisation of usual $\psi$-estimators ($Z$-estimators)
corresponding to a function $\psi\in\Psi(X,\Theta)$, which possesses the property $[C]$.

\begin{Thm}\label{Thm_charac_cont_psi_est}
Let $X$ be a nonempty set, $\Theta$ be a nondegenerate open interval of $\RR$, and $M:\bigcup_{n=1}^\infty X^n\to\Theta$ be a function  such that $\inf M_1(X)=\inf \Theta$ and $\sup M_1(X)=\sup \Theta$.
Then the following two statements are equivalent:
 \vspace{-8pt}
 \begin{enumerate}[(i)]\itemsep=-4pt
  \item There exists a function $\psi\in\Psi[Z,C](X,\Theta)$ such that, for all $n\in\NN$ and $x_1,\ldots,x_n\in X$, it holds that
  \Eq{*}{
  \vartheta_{n,\psi}(x_1,\ldots,x_n)=M_n(x_1,\ldots,x_n).
  }
  In particular, $M_n(x_1,\ldots,x_n)$ is nothing else but the $Z$-estimator based on $x_1,\ldots,x_n$.
 \item The function $M$ possesses the following properties:
       \begin{itemize}
         \item[(a)] Symmetry: $M_n$ is symmetric for each $n\in\NN$,
                i.e., 
                \[
                  M_n(x_1,\ldots,x_n) = M_n(x_{\pi(1)},\ldots,x_{\pi(n)})
                \]   
                for all $x_1,\ldots,x_n \in X$ and each permutation $(\pi(1),\ldots,\pi(n))$ of $(1,\ldots,n)$,
         \item[(b)] Strict internality (strict mean-type property): for each $n,k\in\NN$ and $(x_1,\ldots,x_n)\in X^n$, $(y_1,\ldots,y_k)\in X^k$, we have
               \begin{align*}
                   \min(M_n(x_1,\ldots,x_n),M_k(y_1,\ldots,y_k))
                        &\leq M_{n+k}(x_1,\ldots,x_n,y_1,\ldots,y_k)\\
                        &\leq \max(M_n(x_1,\ldots,x_n),M_k(y_1,\ldots,y_k)),
               \end{align*}
               and if $M_n(x_1,\ldots,x_n)\neq M_k(y_1,\ldots,y_k)$, then both inequalities are strict,
       \item[(c)] Asymptotic idempotency: for all $k\in\NN$, and $x_1,\dots,x_k,y\in X$,
        \[
          \lim_{n\to\infty}M_{kn+1}(\underbrace{x_1,\ldots,x_1}_{n},\dots,\underbrace{x_k,\ldots,x_k}_{n},y) = M_k(x_1,\dots,x_k).
        \]
       \end{itemize}
 \end{enumerate}
\end{Thm}

\begin{proof}
First, assume that assertion (i) holds.
Then, since $\psi$ has the property $[T]$, Theorem \ref{Thm_charac_psi_est} implies that $M$ possesses the properties symmetry and asymptotic idempotency.
Since $\psi$ has the property $[Z]$ as well, the strict internality of $M$ is a consequence of Theorem 2.1 in Barczy and P\'ales \cite{BarPal4}.

Next, assume that assertion (ii) holds.
We divide the forthcoming argument into several steps and substeps.

\textit{Step 1.} Since the strict internality of $M$ implies the internality of $M$, in view of Theorem \ref{Thm_charac_psi_est}, we obtain that there exists a function $\psi^*\in\Psi[T](X,\Theta)$ such that, for all $\ell\in\NN$ and $x_1,\ldots,x_\ell\in X$, the equality $M_\ell(x_1,\ldots,x_\ell)=\vartheta_{\ell,\psi^*}(x_1,\ldots,x_\ell)$ holds.
In what follows, we will define a function $\psi\in\Psi(X,\Theta)$ for which assertion (i) holds, i.e., $\psi$ has the properties $[Z]$ and $[C]$, and the equality $M_\ell(x_1,\ldots,x_\ell)=\vartheta_{\ell,\psi}(x_1,\ldots,x_\ell)$ holds
 for all $\ell\in\NN$ and $x_1,\ldots,x_\ell\in X$.

\textit{Step 2.}
For all $k,\ell\in\NN$, $\bx=(x_1,\dots,x_k)\in X^k$ and $\by=(y_1,\ldots,y_\ell)\in X^\ell$, let us define the function $f_{\bx,\by}:\Theta\setminus\{M_\ell(\by)\}\to\RR$ by
\begin{align}\label{help_cont_function1}
  f_{\bx,\by}(t):=-\frac{\psi^*(x_1,t)+\dots+\psi^*(x_k,t)}{\psi^*(y_1,t)+\dots+\psi^*(y_\ell,t)}, \qquad t\in \Theta\setminus\{M_\ell(\by)\}.
\end{align}
Since $\psi^*$ has the property $[T]$, we have that $\psi^*(y_1,t)+\dots+\psi^*(y_\ell,t) \ne 0$ for all $\by\in X^\ell$
 and $t\in\Theta\setminus\{M_\ell(\by)\}$, and hence $f_{\bx,\by}$ is well-defined.

In what follows, we are going to verify that, for all $k,\ell\in\NN$, $\bx=(x_1,\dots,x_k)\in X^k$ and
 $\by=(y_1,\dots,y_\ell)\in X^\ell$ with $M_k(\bx)<M_\ell(\by)$, the function $f_{\bx,\by}$ given by \eqref{help_cont_function1}
 is positive, increasing and continuous on the open interval $(M_k(\bx),M_\ell(\by))$.

Let $k,\ell\in\NN$, $\bx=(x_1,\dots,x_k)\in X^k$ and $\by=(y_1,\dots,y_\ell)\in X^\ell$ be fixed such that $M_k(\bx)<M_\ell(\by)$.
Since $\psi^*$ has the property $[T]$, we have that $f_{\bx,\by}(t)>0$ for $t\in(M_k(\bx),M_\ell(\by))$.

\textit{Step 2/(a).}\
We prove that $f_{\bx,\by}$ is increasing on $(M_k(\bx),M_\ell(\by))$.
Note that in case of $k=\ell=1$, this property has been established in Barczy and P\'ales \cite[part (iv) of Theorem 1]{BarPal2}.
In case of general $k,\ell\in\NN$, we give an independent proof, which uses different ideas.
To the contrary, assume that $f_{\bx,\by}$ is not increasing on $(M_k(\bx),M_\ell(\by))$.
Then there exist $s,t\in (M_k(\bx),M_\ell(\by))$ with $s<t$ such that $f_{\bx,\by}(s)>f_{\bx,\by}(t)>0$.
Then there exist $n,m\in\NN$ such that
\Eq{*}{
  f_{\bx,\by}(s)>\frac{m}{n}>f_{\bx,\by}(t),
 }
 which be written equivalently in the form
\Eq{*}{
  n\sum_{i=1}^k \psi^*(x_i,s) + m\sum_{j=1}^\ell \psi^*(y_j,s)
  <0<n\sum_{i=1}^k \psi^*(x_i,t) + m \sum_{j=1}^\ell \psi^*(y_j,t).
}
Using that $\psi^*$ has the property $[T]$, these inequalities yield that
\Eq{*}{
  t\leq M_{nk+m\ell}(\underbrace{x_1,\dots,x_1}_n,\dots,\underbrace{x_k,\dots,x_k}_n, \underbrace{y_1,\dots,y_1}_m,\dots,\underbrace{y_\ell,\dots,y_\ell}_m)\leq s,
}
which contradicts the inequality $s<t$.
The contradiction obtained proves that $f_{\bx,\by}$ is increasing on $(M_k(\bx),M_\ell(\by))$.

\textit{Step 2/(b).}\
Now we prove that $f_{\bx,\by}$ is continuous on $(M_k(\bx),M_\ell(\by))$.
To the contrary, assume that $f_{\bx,\by}$ is not continuous at some point of $(M_k(\bx),M_\ell(\by))$.
Then, since $f_{\bx,\by}$ is increasing on $(M_k(\bx),M_\ell(\by))$, there exists a point
 $t_0\in(M_k(\bx),M_\ell(\by))$ such that $f_{\bx,\by}$ has a jump at $t_0$ meaning that the left limit $L$ of $f_{\bx,\by}$
  at $t_0$ is strictly less than the right limit $R$ of $f_{\bx,\by}$ at $t_0$.
Consequently, one can choose $n,m_1,m_2\in\NN$ such that $m_1<m_2$ and $L<\frac{m_1}{n}<\frac{m_2}{n}<R$.
Then, using again that $f_{\bx,\by}$ is increasing on $(M_k(\bx),M_\ell(\by))$, for all $M_k(\bx)<t<t_0<t'<M_\ell(\by)$, we have that
 \[
   f_{\bx,\by}(t) \leq L < \frac{m_1}{n} < \frac{m_2}{n} < R \leq f_{\bx,\by}(t').
 \]
Using that $\sum_{j=1}^\ell \psi^*(y_j,t)>0$ and $\sum_{j=1}^\ell \psi^*(y_j,t')>0$, by rearranging these inequalities, we obtain that, for $\alpha\in\{1,2\}$,
 \begin{align*}
    &n\sum_{i=1}^k \psi^*(x_i,t) + m_\alpha\sum_{j=1}^\ell \psi^*(y_j,t) >0 \quad \text{for all $t\in(M_k(\bx),t_0)$,}\\
    &n\sum_{i=1}^k \psi^*(x_i,t') + m_\alpha\sum_{j=1}^\ell \psi^*(y_j,t') <0 \quad \text{for all $t'\in(t_0,M_\ell(\by))$.}
 \end{align*}
Using that $\psi^*$ has the property $[T]$, these inequalities imply that
 \begin{align*}
   &\vartheta_{nk+\ell m_\alpha,\psi^*}\big(\underbrace{x_1,\dots,x_1}_n,\dots,\underbrace{x_k,\dots,x_k}_n,\underbrace{y_1,\dots,y_1}_{m_\alpha},\dots,\underbrace{y_\ell,\dots,y_\ell}_{m_\alpha} \big) = t_0, \qquad \alpha\in\{1,2\}.
 \end{align*}
Since $\vartheta_{r,\psi^*}(x_1,\ldots,x_r) = M_r(x_1,\ldots,x_r)$ for all $r\in\NN$ and $x_1,\ldots,x_r\in X$, we have
 \begin{align*}
   M_{nk+\ell m_\alpha}\big(\underbrace{x_1,\dots,x_1}_n,\dots,\underbrace{x_k,\dots,x_k}_n, \underbrace{y_1,\dots,y_1}_{m_\alpha},\dots,\underbrace{y_\ell,\dots,y_\ell}_{m_\alpha} \big)   
      = t_0 \in (M_k(\bx),M_\ell(\by))
 \end{align*}
 for $\alpha\in\{1,2\}$, yielding that
 \begin{align}\label{help_folytonossag}
  \begin{split}
  &M_{nk+\ell m_1}\big(\underbrace{x_1,\dots,x_1}_n,\dots,\underbrace{x_k,\dots,x_k}_n, \underbrace{y_1,\dots,y_1}_{m_1},\dots,\underbrace{y_\ell,\dots,y_\ell}_{m_1} \big)  \\
  &\qquad = M_{nk+\ell m_2}\big(\underbrace{x_1,\dots,x_1}_n,\dots,\underbrace{x_k,\dots,x_k}_n, \underbrace{y_1,\dots,y_1}_{m_2},\dots,\underbrace{y_\ell,\dots,y_\ell}_{m_2} \big)
         \in (M_k(\bx),M_\ell(\by)).
 \end{split}
 \end{align}

In what follows, we prove that \eqref{help_folytonossag} cannot be valid.
Recall, that in view of the symmetry and internality of $M$, we have
 \begin{align*}
  M_\ell(\by)
    = M_{\ell(m_2-m_1)}\big(\underbrace{\by,\ldots,\by}_{m_2-m_1}\big)
    = M_{\ell(m_2-m_1)}\big( \underbrace{y_1,\dots,y_1}_{m_2-m_1},\dots,\underbrace{y_\ell,\dots,y_\ell}_{m_2-m_1} \big).
 \end{align*}
Consequently, using the inclusion in \eqref{help_folytonossag} and the symmetry and strict internality of $M$, we get that
 \begin{align*}
    &M_{nk+\ell m_1}\big(\underbrace{x_1,\dots,x_1}_n,\dots,\underbrace{x_k,\dots,x_k}_n, \underbrace{y_1,\dots,y_1}_{m_1},\dots,\underbrace{y_\ell,\dots,y_\ell}_{m_1}  \big)\\
    &\quad  = \min\Big( M_{nk+\ell m_1}\big(\underbrace{x_1,\dots,x_1}_n,\dots,\underbrace{x_k,\dots,x_k}_n, \underbrace{y_1,\dots,y_1}_{m_1},\dots,\underbrace{y_\ell,\dots,y_\ell}_{m_1} \big), M_\ell(\by) \Big)\\
    &\quad = \min\Big( M_{nk+\ell m_1}\big(\underbrace{x_1,\dots,x_1}_n,\dots,\underbrace{x_k,\dots,x_k}_n, \underbrace{y_1,\dots,y_1}_{m_1},\dots,\underbrace{y_\ell,\dots,y_\ell}_{m_1}  \big), \\
    &\phantom{\quad =  \min\Big(\;}
            M_{\ell(m_2-m_1)}\big( \underbrace{y_1,\dots,y_1}_{m_2-m_1},\dots,\underbrace{y_\ell,\dots,y_\ell}_{m_2-m_1} \big)\Big)\\
    &\quad < M_{nk+\ell m_2}\big(\underbrace{x_1,\dots,x_1}_n,\dots,\underbrace{x_k,\dots,x_k}_n,\underbrace{y_1,\dots,y_1}_{m_2},\dots,\underbrace{y_\ell,\dots,y_\ell}_{m_2}  \big),
 \end{align*}
which contradicts the equality in \eqref{help_folytonossag}.
The contradiction obtained proves the continuity of $f_{\bx,\by}$ on $(M_k(\bx),M_\ell(\by))$.

\textit{Step 3.}
We show that, for all $k,\ell\in\NN$, $\bx=(x_1,\dots,x_k)\in X^k$ and $\by=(y_1,\dots,y_\ell)\in X^\ell$ with $M_\ell(\by)<M_k(\bx)$,
 the function $f_{\bx,\by}$ given by \eqref{help_cont_function1} is decreasing and continuous on the open interval $(M_\ell(\by), M_k(\bx))$.
In what follows, let $k,\ell\in\NN$, $\bx=(x_1,\dots,x_k)\in X^k$ and $\by=(y_1,\dots,y_\ell)\in X^\ell$ be fixed such that $M_\ell(\by)<M_k(\bx)$.
By Step 2, the function $f_{\by,\bx}$ is positive, increasing
 and continuous on $(M_\ell(\by), M_k(\bx))$.
Further, on the interval $(M_\ell(\by), M_k(\bx))$, we have that $f_{\bx,\by}=1/f_{\by,\bx}$,
 which shows that $f_{\bx,\by}$ is decreasing and continuous on $(M_\ell(\by), M_k(\bx))$.

\textit{Step 4.}
We show that, for all $k,\ell\in\NN$, $\bx=(x_1,\dots,x_k)\in X^k$ and $\by=(y_1,\dots,y_\ell)\in X^\ell$ with $M_k(\bx)<M_\ell(\by)$, the function $f_{\bx,\by}$ given by \eqref{help_cont_function1} is continuous on its entire domain $\Theta\setminus\{M_\ell(\by)\}$.
In Step 2, we have already proved that $f_{\bx,\by}$ is continuous on $(M_k(\bx),M_\ell(\by))$.
In what follows, let $k,\ell\in\NN$, $\bx=(x_1,\dots,x_k)\in X^k$ and $\by=(y_1,\dots,y_\ell)\in X^\ell$ be fixed such that $M_k(\bx)<M_\ell(\by)$.

\textit{Step 4/(a).}
We show that $f_{\bx,\by}$ is continuous at $M_k(\bx)$.
To show this, choose an element $z\in X$ such that $M_1(z)<M_k(\bx)$.
Since $\inf M_1(X) = \inf \Theta$, $M_k(\bx)\in \Theta$, and $\Theta$ is open, such an element $z\in X$ can be indeed chosen.
Then, by the strict internality of $M$, we have that $M_1(z)<M_{k+1}(z,\bx)<M_k(\bx)$.
Hence, by Step 2, we get that $f_{z,\by}$ is continuous on $(M_1(z),M_\ell(\by))$, and $f_{(z,\bx),\by}$ is continuous on $(M_{k+1}(z,\bx),M_\ell(\by))$.
In particular, since $M_k(\bx)\in (M_1(z),M_\ell(\by))$ and $M_k(\bx)\in (M_{k+1}(z,\bx),M_\ell(\by))$,
we have that $f_{z,\by}$ and $f_{(z,\bx),\by}$ are continuous at $M_k(\bx)$.
Using the decomposition
 \begin{align*}
  f_{\bx,\by}(t) & = -\frac{\psi^*(x_1,t)+ \cdots + \psi^*(x_k,t)}{\psi^*(y_1,t)+\cdots + \psi^*(y_\ell,t)}
                 = -\frac{(\psi^*(z,t)+ \psi^*(x_1,t) + \cdots + \psi^*(x_k,t)) - \psi^*(z,t)}{\psi^*(y_1,t)+\cdots + \psi^*(y_\ell,t)} \\
               & = f_{(z,\bx),\by}(t) - f_{z,\by}(t), \qquad t\in\Theta\setminus\{ M_\ell(\by)\},
 \end{align*}
 and the facts that $f_{(z,\bx),\by}$ and $f_{z,\by}$ are continuous at $M_k(\bx)$,
 we obtain that $f_{\bx,\by}$ is also continuous at $M_k(\bx)$.

\textit{Step 4/(b).}
We show that $f_{\bx,\by}$ is continuous at $t$, where $t<M_k(\bx)$ and $t\in\Theta$. Let $t\in\Theta$ be fixed such that $t<M_k(\bx)$.
Since $\Theta$ is open, $t\in\Theta$ and $\inf M_1(X) = \inf\Theta$, we can choose an element $z\in X$ such that $M_1(z)<t$.
Due to the fact that $M_1(z) < t < M_k(\bx) < M_\ell(\by)$, by Step 2, we get that $f_{z,\by}$ is continuous on $(M_1(z),M_\ell(\by))$
 and $f_{z,\bx}$ is continuous on $(M_1(z),M_k(\bx))$.
In particular, $f_{z,\by}$ and $f_{z,\bx}$ are continuous at $t$.
Since $\psi^*(z,s)<0$, $\sum_{i=1}^k\psi^*(x_i,s)>0$ and $\sum_{j=1}^\ell\psi^*(y_j,s)>0$ for $s\in(M_1(z),M_k(\bx))$, we can consider the decomposition
 \begin{align*}
 f_{\bx,\by}(s)
          & = -\frac{\psi^*(x_1,s)+\cdots +\psi^*(x_k,s)}{\psi^*(y_1,s)+\cdots + \psi^*(y_\ell,s)}\\
            &= -\frac{\psi^*(x_1,s)+\cdots +\psi^*(x_k,s)}{\psi^*(z,s)}\cdot \frac{\psi^*(z,s)}{\psi^*(y_1,s)+\cdots + \psi^*(y_\ell,s)} \\
          & = - \frac{f_{z,\by}(s)}{f_{z,\bx}(s)},
           \qquad s\in(M_1(z),M_k(\bx)).
 \end{align*}
Using that $t\in(M_1(z),M_k(\bx))$,  this implies that $f_{\bx,\by}$ is continuous at $t$.

\textit{Step 4/(c).}
Finally, we show that $f_{\bx,\by}$ is continuous at $t$, where $t>M_\ell(\by)$ and $t\in\Theta$. Let $t\in\Theta$ be fixed such that $t>M_\ell(\by)$.
Since $\Theta$ is open, $t\in\Theta$ and $\sup M_1(X) = \sup\Theta$, we can choose an element $z\in X$ such that $M_1(z)>t$.
Due to the fact that $M_k(\bx) < M_\ell(\by) < t < M_1(z)$, by Step 2, we get that $f_{\bx,z}$ is continuous on $(M_k(\bx),M_1(z))$
 and $f_{\by,z}$ is continuous on $(M_\ell(\by),M_1(z))$.
In particular, $f_{\bx,z}$ and $f_{\by,z}$ are continuous at $t$.
Since $\psi^*(z,s)>0$ and $\sum_{j=1}^\ell \psi^*(y_j,s)<0$ for $s\in(M_\ell(\by),M_1(z))$, we can consider the decomposition
 \begin{align*}
 f_{\bx,\by}(s)&= -\frac{\psi^*(x_1,s)+\cdots +\psi^*(x_k,s)}{\psi^*(y_1,s)+\cdots + \psi^*(y_\ell,s)}\\
             &= -\frac{\psi^*(x_1,s)+\cdots +\psi^*(x_k,s)}{\psi^*(z,s)}\cdot \frac{\psi^*(z,s)}{\psi^*(y_1,s)+\cdots + \psi^*(y_\ell,s)} \\
             &= - \frac{f_{\bx,z}(s)}{f_{\by,z}(s)},
              \qquad s\in(M_\ell(\by),M_1(z)).
 \end{align*}
Using that $t\in(M_\ell(\by),M_1(z))$,  this implies that $f_{\bx,\by}$ is continuous at $t$.

\textit{Step 5.}\
We show that, for all $k,\ell\in\NN$, $\bx=(x_1,\dots,x_k)\in X^k$ and $\by=(y_1,\dots,y_\ell)\in X^\ell$ with $M_\ell(\by)<M_k(\bx)$, the function $f_{\bx,\by}$ given by \eqref{help_cont_function1} is continuous on its entire domain $\Theta\setminus\{M_\ell(\by)\}$).
In what follows, let $k,\ell\in\NN$, $\bx=(x_1,\dots,x_k)\in X^k$ and $\by=(y_1,\dots,y_\ell)\in X^\ell$ be fixed such that $M_\ell(\by)<M_k(\bx)$.

\textit{Step 5/(a).}\
We show that $f_{\bx,\by}$ is continuous at $M_k(\bx)$.
To show this, choose an element $z\in X$ such that $M_k(\bx)<M_1(z)$.
Since $\sup M_1(X) = \sup \Theta$, $M_k(\bx)\in \Theta$, and $\Theta$ is open, such an element $z\in X$ can be indeed chosen.
Then, by the strict internality of $M$, we have that $M_\ell(\by)<M_k(\bx)<M_{k+1}(\bx,z)<M_1(z)$.
Hence, by Step 3, we get that $f_{z,\by}$ is continuous on $(M_\ell(\by),M_1(z))$,
 and $f_{(\bx,z),\by}$ is continuous on $(M_\ell(\by),M_{k+1}(z,\bx))$.
In particular, $f_{z,\by}$ and $f_{(\bx,z),\by}$ are continuous at $M_k(\bx)$.
Using the decomposition
 \begin{align*}
  f_{\bx,\by}(t) & = -\frac{\psi^*(x_1,t)+ \cdots + \psi^*(x_k,t)}{\psi^*(y_1,t)+\cdots + \psi^*(y_\ell,t)}
                 = -\frac{(\psi^*(z,t)+ \psi^*(x_1,t) + \cdots + \psi^*(x_k,t)) - \psi^*(z,t)}{\psi^*(y_1,t)+\cdots + \psi^*(y_\ell,t)} \\
               & = f_{(z,\bx),\by}(t) - f_{z,\by}(t), \qquad t\in\Theta\setminus\{ M_\ell(\by)\},
 \end{align*}
 and the facts that $M_k(\bx) \ne M_\ell(\by)$, $f_{(\bx,z),\by} = f_{(z,\bx),\by}$  and that $f_{(\bx,z),\by}$ and $f_{z,\by}$ are continuous at $M_k(\bx)$,
 we obtain that $f_{\bx,\by}$ is also continuous at $M_k(\bx)$.

\textit{Step 5/(b).}\
We show that $f_{\bx,\by}$ is continuous on $\Theta\setminus\{ M_k(\bx),M_\ell(\by)\}$.
By Step 4, the function $f_{\by,\bx}$ is continuous on $\Theta\setminus\{ M_k(\bx)\}$.
Since
 \[
   f_{\bx,\by}(t) = \frac{1}{ f_{\by,\bx}(t)}, \qquad t\in \Theta\setminus\{ M_k(\bx),M_\ell(\bx)\},
 \]
 we can conclude that $f_{\bx,\by}$ is continuous on $\Theta\setminus\{ M_k(\bx),M_\ell(\bx)\}$.

\textit{Step 6.}\
Let us choose $u,v\in X$ such that $M_1(u)\neq M_1(v)$.
Since, by the assumptions, $M_1$ is not a constant function, such elements $u$ and $v$ can be choosen.

Let $\psi:X\times \Theta\to\RR$ be defined by
 \[
   \psi(x,t):=\frac{\psi^*(x,t)}{|\psi^*(u,t)|+|\psi^*(v,t)|}, \qquad x\in X, \;\; t\in\Theta.
 \]
Then $\psi$ is well-defined, since $\psi^*$ has the property $[T]$ and hence if $t\in\Theta$ is such that $t\neq M_1(u)$, then $|\psi^*(u,t)|>0$ and if $t=M_1(u)$, then $|\psi^*(v,t)|>0$ (due to the facts that $M_1(u) \ne M_1(v)$ and $|\psi^*(v,s)|>0$ for all $s\in\Theta\setminus\{M_1(v)\}$).

\textit{Step 7.}\
We show that $\psi\in\Psi[Z,C](X,\Theta)$, and that the equality
 $\vartheta_{n,\psi}(x_1,\ldots,x_n) = M_n(x_1,\ldots,x_n)$ holds for all $n\in\NN$ and $x_1,\ldots,x_n\in X$.

\textit{Step 7/(a).}\ For all $n\in\NN$ and $x_1,\ldots,x_n\in X$, we have
 \[
  \sum_{i=1}^n \psi(x_i,t) = \frac{1}{|\psi^*(u,t)|+|\psi^*(v,t)|} \sum_{i=1}^n \psi^*(x_i,t), \qquad t\in\Theta.
 \]
Using that $|\psi^*(u,t)|+|\psi^*(v,t)|>0$, $t\in\Theta$, and $\psi^*$ has the property $[T]$ (see Step 1), it follows that $\psi$ has the property $[T]$ as well, and $\vartheta_{n,\psi}(x_1,\ldots,x_n) = \vartheta_{n,\psi^*}(x_1,\ldots,x_n)$, yielding that
 $\vartheta_{n,\psi}(x_1,\ldots,x_n) = M_n(x_1,\ldots,x_n)$ for all $n\in\NN$ and $x_1,\ldots,x_n\in X$, as desired.

\textit{Step 7/(b).}\
We show that $\psi$ has the properties $[Z]$ and $[C]$.
Since $\psi$ has the property $[T]$ (see Step 7/(a)), it is enough to check that $\psi$ has the property $[C]$.

If $t\in\Theta$ is such that $t\neq M_1(u)$, then $\psi^*(u,t)\neq0$ and we can write
\Eq{*}{
  \psi(x,t):=\frac{\dfrac{\psi^*(x,t)}{|\psi^*(u,t)|}}{1+\dfrac{|\psi^*(v,t)|}{|\psi^*(u,t)|}}
  =\begin{cases}
    - \dfrac{f_{x,u}(t)}{1+|f_{v,u}(t)|} &\mbox{if $t<M_1(u)$},\\[4mm]
   \dfrac{f_{x,u}(t)}{1+|f_{v,u}(t)|} &\mbox{if $t>M_1(u)$},
   \end{cases}
   \qquad x\in X. 
  }
By Steps 4 and 5, the functions $f_{x,u}$ and $f_{v,u}$ are continuous on $\Theta\setminus\{M_1(u)\}$.
This shows that, for all $x\in X$, the map $t\mapsto\psi(x,t)$ is continuous on the set $\Theta\setminus\{M_1(u)\}$.

Similarly, if $t\in\Theta\setminus\{M_1(v)\}$, then $\psi^*(v,t)\neq0$ and we can write
\Eq{*}{
  \psi(x,t):=\frac{\dfrac{\psi^*(x,t)}{|\psi^*(v,t)|}}{\dfrac{|\psi^*(u,t)|}{|\psi^*(v,t)|}+1}
  =\begin{cases}
   -\dfrac{f_{x,v}(t)}{|f_{u,v}(t)|+1} &\mbox{if $t<M_1(v)$},\\[4mm]
   \dfrac{f_{x,v}(t)}{|f_{u,v}(t)|+1} &\mbox{if $t>M_1(v)$},
   \end{cases}
   \qquad x\in X.
}
By Steps 4 and 5, the functions $f_{x,v}$ and $f_{u,v}$ are continuous on $\Theta\setminus\{M_1(v)\}$.
This shows that, for all $x\in X$, the map $t\mapsto\psi(x,t)$ is continuous on the set $\Theta\setminus\{M_1(v)\}$.

The above two continuity properties and the fact that $M_1(u)\neq M_1(v)$ imply that the map $t\mapsto\psi(x,t)$ is continuous on $(\Theta\setminus\{M_1(u)\})\cup (\Theta\setminus\{M_1(v)\}) =\Theta$,
and hence $\psi$ possesses the property $[C]$.
\end{proof}

Concerning the proof of Theorem \ref{Thm_charac_cont_psi_est}, we mention that
it is not a simple application or modification of that of Theorem \ref{Thm_charac_psi_est}, most of it (Steps 2-7 when we verify that part (ii) implies part (i)) is devoted to find an appropriate version of the function $\psi^*$ appearing in Step 1 such that it has the properties $[Z]$ and $[C]$.
In the forthcoming Remark \ref{Rem_Thm_31_kieg}, we also point out that, one can show in a somewhat simpler manner that the function $\psi^*$ in question has the property $[Z]$.
However, we call the attention that, in general, we cannot prove that $\psi^*$ has the property $[C]$ not even when we know that it has the property $[Z]$.
This explains the reason for introducing a modified version of $\psi^*$ in Step 6 in the proof of Theorem \ref{Thm_charac_cont_psi_est}.
We emphasize that the proof given in Remark \ref{Rem_Thm_31_kieg} is not needed for the proof of Theorem \ref{Thm_charac_cont_psi_est},
 it just deepens the understanding in the sense that it highlights that the property $[Z]$ can be proved for the function $\psi^*$ appearing in Step 1 (in the proof of Theorem \ref{Thm_charac_cont_psi_est}).

\begin{Rem}\label{Rem_Thm_31_kieg}
We give a proof of the fact that the function $\psi^*$  appearing in Step 1 in the proof of Theorem \ref{Thm_charac_cont_psi_est}
 possesses the property $[Z]$. However, we emphasize again that, in general, we cannot prove that $\psi^*$ has the property $[C]$ not even when we know that it has the property $[Z]$.
Let $k\in\NN$ and $\bx=(x_1,\dots,x_k)\in X^k$ be fixed.
Taking into account that $\psi^*$ has the property $[T]$ and $\vartheta_{k,\psi^*}(\bx) = M_k(\bx)$,
 we need to verify that $\sum_{i=1}^k\psi^*(x_i,M_k(\bx))=0$.
Since $\sup M_1(X) = \sup \Theta$, $M_k(\bx)\in \Theta$ and $\Theta$ is open,
 one can choose an element $y\in X$ such that $M_k(\bx)<M_1(y)$.
Consider the function $f_{\bx,y}:\Theta\setminus\{M_1(y)\}\to\RR$ given in \eqref{help_cont_function1}:
 \[
    f_{\bx,y}(t)=-\frac{\psi^*(x_1,t)+\dots + \psi^*(x_k,t)}{\psi^*(y,t)},\qquad t\in \Theta\setminus\{M_1(y)\}.
 \]
Using that $\sum_{i=1}^k\psi^*(x_i,t)>0$ and $\psi^*(y,t)>0$ for $t\in(\inf \Theta, M_k(\bx))$, we have that
 \begin{align}\label{help_Z_1}
  f_{\bx,y}(t)<0, \qquad t\in(\inf \Theta, M_k(\bx)).
 \end{align}
Similarly, since $\sum_{i=1}^k\psi^*(x_i,t)<0$ and $\psi^*(y,t)>0$ for $t\in (M_k(\bx),M_1(y))$, we have that
 \begin{align}\label{help_Z_2}
 f_{\bx,y}(t)>0, \qquad t\in (M_k(\bx),M_1(y)).
 \end{align}
By Step 4 in the proof of Theorem \ref{Thm_charac_cont_psi_est},
 $f_{\bx,y}$ is continuous (on $\Theta\setminus\{M_1(y)\}$), and hence the limits
 \[
   \lim_{t\uparrow M_k(\bx)} f_{\bx,y}(t) \qquad \text{and} \qquad \lim_{t\downarrow M_k(\bx)} f_{\bx,y}(t)
 \]
 exist and coincide with $f_{\bx,y}(M_k(\bx))$.
By \eqref{help_Z_1} and \eqref{help_Z_2}, we have
 \[
   f_{\bx,y}(M_k(\bx)) = \lim_{t\uparrow M_k(\bx)} f_{\bx,y}(t)\leq 0
     \qquad \text{and}\qquad
   f_{\bx,y}(M_k(\bx))= \lim_{t\downarrow M_k(\bx)} f_{\bx,y}(t)\geq 0.
 \]
Consequently, $f_{\bx,y}(M_k(\bx))=0$, which yields that $\psi^*(\bx,M_k(\bx))=0$, as desired.
\proofend
\end{Rem}

\section*{Acknowledgments}
We acknowledge the valuable suggestions from the referees.

\section*{Conflict of interest}
On behalf of all authors, the corresponding author states that there is no conflict of interest.

\bibliographystyle{plain}

\end{document}